\documentclass[12pt]{amsart}
\usepackage{amsmath,amsthm,amsfonts,amssymb,xcolor}
\usepackage{graphicx}
\usepackage{mathrsfs}
\usepackage{enumitem}
\usepackage{etoolbox}
\apptocmd{\sloppy}{\hbadness 10000\relax}{}{}
\apptocmd{\sloppy}{\vbadness 10000\relax}{}{}
\usepackage{hyperref}
\usepackage[letterpaper,margin=1.1in]{geometry}

\numberwithin{equation}{section}
\theoremstyle{plain}

\newtheorem{theorem}{Theorem}[section]
\newtheorem{proposition}[theorem]{Proposition}
\newtheorem{corollary}[theorem]{Corollary}
\newtheorem{lemma}[theorem]{Lemma}

\theoremstyle{definition}

\def\RR{\mathbb{R}}
\def\ZZ{\mathbb{Z}}

\newcommand{\diam}{\mathop\mathrm{diam}\nolimits}
\newcommand{\dist}{\mathop\mathrm{dist}\nolimits}
\newcommand{\side}{\mathop\mathrm{side}\nolimits}
\newcommand{\Span}{\mathop\mathrm{span}\nolimits}

\newcommand{\Graph}{\mathop\mathsf{Graph}\nolimits}
\newcommand{\Defect}{\mathop\mathsf{Defect}\nolimits}
\newcommand{\gap}{\mathop\mathrm{gap}\nolimits}

\newcommand{\res}{\hbox{ {\vrule height .22cm}{\leaders\hrule\hskip.2cm} }}

\newcommand{\Haus}{\mathcal{H}}

\newcommand{\Leaves}{\mathop\mathsf{Leaves}\nolimits}
\newcommand{\Top}{\mathop\mathsf{Top}\nolimits}
\newcommand{\spt}{\mathop\mathrm{spt}\nolimits}

\newcommand{\excess}{\mathop\mathrm{excess}\nolimits}

\begin{document}

\title{Radon measures and Lipschitz graphs}

\author{Matthew Badger \and Lisa Naples}
\thanks{The authors were partially supported by NSF DMS grant 1650546.}
\date{December 16, 2020}
\subjclass[2010]{Primary 28A75}
\keywords{Radon measure, Lipschitz graph, rectifiability, doubling, cone approximation}

\address{Department of Mathematics\\ University of Connecticut\\ Storrs, CT 06269-1009}
\email{matthew.badger@uconn.edu}
\address{Department of Mathematics\\ University of Connecticut\\ Storrs, CT 06269-1009}
\email{lisa.naples@uconn.edu}

\begin{abstract}For all $1\leq m\leq n-1$, we investigate the interaction of locally finite measures in $\RR^n$ with the family of $m$-dimensional Lipschitz graphs. For instance, we characterize Radon measures $\mu$, which are carried by Lipschitz graphs in the sense that there exist graphs $\Gamma_1,\Gamma_2,\dots$ such that $\mu(\RR^n\setminus\bigcup_1^\infty\Gamma_i)=0$, using only countably many evaluations of the measure. This problem in geometric measure theory was classically studied within smaller classes of measures, e.g.~for the restrictions of $m$-dimensional Hausdorff measure $\Haus^m$ to $E\subseteq \RR^n$ with $0<\Haus^m(E)<\infty$. However, an example of Cs\"ornyei, K\"{a}enm\"{a}ki, Rajala, and Suomala shows that classical methods are insufficient to detect when a general measure charges a Lipschitz graph. To develop a characterization of Lipschitz graph rectifiability for arbitrary Radon measures, we look at the behavior of coarse doubling ratios of the measure on dyadic cubes that intersect conical annuli. This extends a characterization of graph rectifiability for pointwise doubling measures by Naples by mimicking the approach used in the characterization of Radon measures carried by rectifiable curves by Badger and Schul.\end{abstract}

\maketitle

\tableofcontents

\section{Introduction}\label{sec:intro}

A general goal in geometric measure theory that has not yet been fully achieved is to understand in a systematic way how generic measures on a metric space interact with a prescribed family of sets in the space, e.g.~rectifiable curves, smooth submanifolds, etc. For instance, we could ask: Is a measure positive on some set in the family? Do there exist countably many sets in the family whose union captures all of the mass of the measure? For Hausdorff measures and measures with \emph{a priori} bounds on the asymptotic densities of the measure, much progress has been made. A description of this work as it stood at the end of the last century can be found can be found in \cite{Mattila}. Newer developments in the theory of rectifiability of absolutely continuous measures include \cite{AT15, Tolsa-n, TT-rect, ENV, Ghinassi, Goering-rect, Dabrowski-n,Dabrowski-s}. An alternative regularity condition that is usually \emph{a priori} weaker than upper and lower density bounds is asymptotic control on how much the measure grows when the radius of a ball is doubled. Recent investigations on the rectifiability of doubling measures include \cite{ADT1,ADT2,ATT-alphas,Naples-TST}. For Radon measures, which in general do not possess good bounds on density or doubling, the situation is far less understood and examples show that classical geometric measure theory methods are not strong enough to detect when measures charge Lipschitz images of Euclidean subspaces or graphs of Lipschitz functions \cite{MM1988}, \cite{CKRS}, \cite{GKS}, \cite{MO-curves}, \cite{Tolsa-betas}. On the positive side, we now possess a complete description of the interaction of an arbitrary Radon measure in Euclidean space with rectifiable curves \cite{BS3}. This advance required a thorough understanding of the geometry of subsets of rectifiable curves in Hilbert spaces \cite{Jones-TST,Ok-TST,Schul-Hilbert} and further blending geometric measure theory with techniques from modern harmonic analysis. For a longer overview of these and other related developments on generalized rectifiability of measures, including fractional and higher-order rectifiability,  see the survey \cite{ident}.

In this paper, we obtain a complete description of the interaction of Radon measures in $\RR^n$ with graphs of Lipschitz functions over $m$-dimensional subspaces for all $1\leq m\leq n-1$. Moreover, the characterization of Lipschitz graph rectifiability that we identify depends only on the value of the measure on dyadic cubes below a fixed generation. The key insight is that to construct Lipschitz graphs that charge a measure, one must be able to equitably distribute mass which appears in bad cones. For a detailed statement, see Theorem \ref{t:main} and the definition of cone defect. The connection between Lipschitz graphs and the geometry of measures have been studied for over ninety years, appearing in foundational work on the structure of Hausdorff measures in the plane \cite{Bes28,Bes38,MR44} and higher-dimensional Euclidean space \cite{Fed47}. Radon measures on smooth Lipschitz graphs supply a model for generalized surfaces in connection with Plateau's problem \cite{Almgren-survey,Plateau-again}. Beyond the domain of geometric measure theory, understanding  rectifiability of measures with respect to Lipschitz graphs is crucial in the study of boundedness of singular integral operators \cite{CMM-Lipschitz,DS91,DS93,David-Semmes-conjecture} and absolute continuity of harmonic measure on rough domains \cite{David-Jerison,Badger-nullsets,7author,AAM-advances}.

\subsection*{Cones and Lipschitz Graph Rectifiability} Throughout the paper, we fix integer dimensions $1\leq m\leq n-1$, where $n$ denotes the dimension of ambient space and $m$ denotes the dimension of a Lipschitz graph. A \emph{bad cone} $X=X(V,\alpha)$ is a set of the form $$X=\{x\in\RR^n:\dist(x,V)>\alpha\dist(x,V^\perp)\},$$ where $V\in G(n,m)$ is any $m$-dimensional subspace of $\RR^n$, $V^\perp\in G(n,n-m)$ denotes its orthogonal complement, and $\alpha\in(0,\infty)$. In other words, $X$ is the set of points which are relatively closer to $V^\perp$ than $V$. We exclude the degenerate case $\alpha=0$, which corresponds to $X=\RR^n\setminus V$. For every $x\in \RR^n$ and bad cone $X$, we let $X_x=x+X=\{x+y:y\in X\}$ denote the translate of $X$ with center $x$. The importance of this family of cones is that they yield a perfect test to determine when a set is contained in a Lipschitz graph.

\begin{lemma}[Geometric Lemma] \label{geometric-lemma} Let $V\in G(n,m)$, $\alpha\in(0,\infty)$, $X=X(V,\alpha)$, and $E\subseteq\RR^n$ be nonempty. There exists a Lipschitz function $f:V\rightarrow V^\perp$ with Lipschitz constant at most $\alpha$ such that $E$ is contained in $$\Graph(f)=\{(x,f(x)):x\in V\}\subseteq V\times V^\perp=\RR^n$$ if and only if $E\cap X_x=\emptyset$ for all $x\in E$. The conclusion also holds when $\alpha=0$.\end{lemma}

\begin{proof} Unwind the definitions or see e.g.~the proof in \cite[Lemma 4.7]{DeLellis}.\end{proof}

We adopt the convention that a \emph{Radon measure} $\mu$ on $\RR^n$ is a Borel regular outer measure that is finite on compact sets. At one extreme, given a nonempty family $\mathcal{F}$ of Borel sets, we say that $\mu$ is \emph{carried by $\mathcal{F}$} if there exist $F_1,F_2,\dots\in\mathcal{F}$ such that $\mu(\RR^n\setminus \bigcup_1^\infty F_i)=0$. At the other extreme, we say that $\mu$ is \emph{singular to $\mathcal{F}$} provided $\mu(F)=0$ for every $F\in\mathcal{F}$. For any Radon measure $\mu$ and Borel set $E\subseteq\RR^n$, the \emph{restriction} $\mu\res E$ defined by the rule $\mu\res E(A)=\mu(E\cap A)$ for all sets $A\subseteq\RR^n$ is again a Radon measure. The \emph{support} $\spt\mu$ of a Radon measure is the smallest closed set such that $\mu(\RR^n\setminus\spt\mu)=0$.

A Radon measure $\mu$ on $\RR^n$ is \emph{Lipschitz graph rectifiable} of dimension $m$ if $\mu$ is carried by \emph{$m$-dimensional Lipschitz graphs}, i.e.~graphs of Lipschitz functions $f:V\rightarrow V^\perp$ over subspaces $V\in G(n,m)$. For example, let $\Gamma_1,\Gamma_2,\dots$ be a sequence of Lipschitz graphs in $\RR^n$ with uniformly bounded Lipschitz constants, let $\mu_i=\Haus^m\res \Gamma_i$ denote the restriction of the $m$-dimensional Hausdorff measure $\mathcal{H}^m$ to $\Gamma_i$, and let $c_1,c_2,\dots\in(0,\infty)$ be a sequence of weights with $\sum_1^\infty c_i<\infty$. Then $\mu=\sum_1^\infty c_i\mu_i$ is an $m$-dimensional Lipschitz graph rectifiable Radon measure on $\RR^n$ with support equal to $\overline{\bigcup_1^\infty\Gamma_i}$. In particular, there exist Lipschitz graph rectifiable measures with $\spt\mu=\RR^n$.

The following classical criterion for Lipschitz graph rectifiability is due to Federer \cite[Theorem 4.7]{Fed47}. The theorem only applies to Radon measures satisfying the upper density bounds $0<\limsup_{r\downarrow 0} r^{-m}\mu(B(x,r))<\infty$ $\mu$-a.e. Such measures are strongly $m$-dimensional in the  sense that $\mu$ is carried by sets of finite Hausdorff measure $\mathcal{H}^m$ and singular to sets of zero $\mathcal{H}^m$ measure; e.g.~this can be shown using \cite[Theorem 6.9]{Mattila}.

\begin{theorem}[Federer] \label{federer-theorem} Let $\mu$ be a Radon measure on $\RR^n$. If for $\mu$-a.e.~$x\in\RR^n$, there exists a bad cone $X=X(V_x,\alpha_x)$ such that \begin{equation}\label{federer-condition}\limsup_{r\downarrow 0} \frac{\mu(X_x\cap B(x,r))}{r^m} < \frac{\alpha_x^m}{2\cdot 100^m} \limsup_{r\downarrow 0} \frac{\mu(B(x,r))}{r^m}<\infty,\end{equation} then $\mu$ is carried by $m$-dimensional Lipschitz graphs. \end{theorem}

For general locally finite measures, there is no uniform comparison between the measure of a ball and the radius of the ball. Thus, it is natural to try to replace the normalizing factor $r^m$ in Federer's condition \eqref{federer-condition} with $\mu(B(x,r))$. In this direction, Naples obtained a characterization of Lipschitz graph rectifiability for measures in Hilbert space with pointwise bounded asymptotic doubling. See \cite[Theorem D]{Naples-TST}.

\begin{theorem}[Naples] \label{t:naples} Let $\mu$ be a Radon measure on $\RR^n$ or a locally finite Borel regular outer measure on the Hilbert space $\ell_2$. Assume that $\mu$ is pointwise doubling, i.e.~$$\limsup_{r\downarrow 0} \frac{\mu(B(x,2r))}{\mu(B(x,r))}<\infty\quad\text{at $\mu$-a.e.~$x$}.$$ Then $\mu$ is carried by $m$-dimensional Lipschitz graphs if and only if for $\mu$-a.e.~$x$ there exists a bad cone $X=X(V_x,\alpha_x)$ such that \begin{equation}\label{eq:cone-point} \lim_{r\downarrow 0} \frac{\mu(X_x\cap B(x,r))}{\mu(B(x,r))}=0.\end{equation}  \end{theorem}

The restriction to pointwise doubling measures in Theorem \ref{t:naples} is crucial. For general Radon measures in Euclidean space, which may be non-doubling, it is impossible to characterize Lipschitz graph rectifiability using condition \eqref{eq:cone-point} by an example of Cs\"ornyei, K\"{a}enm\"{a}ki, Rajala, and Suomala. See  \cite[Example 5.5]{CKRS}.

\begin{theorem}[Cs\"ornyei \emph{et al.}] There exists a non-zero Radon measure $\mu$ on $\RR^2$ and $V\in G(2,1)$ such that for all $\alpha>0$, condition \eqref{eq:cone-point} holds for $\mu$ and $X=X(V,\alpha)$ at $\mu$-a.e.~$x\in\RR^2$, and $\mu$ is singular to 1-dimensional Lipschitz graphs.\end{theorem}

In a recent preprint \cite{Dabrowski-cones}, Dabrowski independently announced a characterization of $m$-dimensional Lipschitz graph rectifiable measures, which are absolutely continuous with respect to $\mathcal{H}^m$. The characterization is in terms of a Dini condition on conical densities of the form $r^{-m}\mu(X\cap B(x,r))$ and requires certain \emph{a priori} bounds on the lower and upper $m$-dimensional density on $\mu$; for related examples, see \cite{Dabrowski-examples}. In \cite{DNOI}, Del Nin and Obinna Idu supply an extension of Theorem \ref{federer-theorem} to $C^{1,\alpha}$ graphs.

\subsection*{Conical Defect} Our goal is to promote the characterization of subsets of Lipschitz graphs given by the geometric lemma to a characterization of Radon measures carried by Lipschitz graphs. Motivated by the characterization of Radon measures carried by rectifiable curves \cite{BS3}, we follow \cite[Remark 2.10]{ident} and design an anisotropic version of the geometric lemma for measures. For any nonempty set $Q\subseteq\RR^n$, let $X_Q=\bigcup_{x\in Q}X_x$ denote the union of bad cones centered on $Q$. In particular, suppose that $Q$ is a (half-open) dyadic cube, i.e.~a set of the form$$Q=\left[\frac{j_1}{2^k},\frac{j_1+1}{2^k}\right)\times\cdots\times \left[\frac{j_n}{2^k},\frac{j_n+1}{2^k}\right),\quad k,j_1,\dots,j_n\in\ZZ.$$ We denote the side length $2^{-k}$ of $Q$ by $\side Q$. Let $x_Q$ denote the geometric center of $Q$. For each $X=X(V,\alpha)$ with $\alpha\in(0,\infty)$, let $r_{Q,X}>0$ be sufficiently large such that if $R$ is a dyadic cube of the same generation as $Q$ and $R$ intersects the \emph{conical annulus} $A_{Q,X}$, $$A_{Q,X}=X_Q\cap B(x_Q,r_{Q,X})\setminus U(x_Q,r_{Q,X}/3),$$ then $R\cap B(x_Q,r_{Q,X}/4)=\emptyset$ and $\gap(R,X_x(V,\alpha/2)^c)\geq \diam Q$ for all $x\in Q$. Here $B(x,r)$ and $U(x,r)$ denote the closed and open balls centered at $x$ with radius $r$, respectively, $\diam Q$ denotes the diameter of $Q$, $S^c=\RR^n\setminus S$ for every set $S\subseteq\RR^n$, and $\gap(S,T)=\inf\{|s-t|:s\in S,\,t\in T\}$ for all nonempty sets $S,T\subseteq\RR^n$. (In harmonic analysis, $\gap(S,T)$ is often denoted by $\dist(S,T)$, but because the gap between sets fails the triangle inequality, we believe it should not be called a distance; our terminology comes from variational analysis, see e.g.~\cite{Beer}.) By Lemma \ref{bounded-geometry} below, we may choose $$r_{Q,X}=81\sqrt{n}\max(\alpha,1/\alpha)\side Q.$$ Define the \emph{discretized conical annulus} $\Delta^*_{Q,X}$ to be the set of all such $R$, i.e.~$R\in\Delta^*_{Q,X}$ if and only if $R$ is a dyadic cube, $\side R=\side Q$, and $R $ has nonempty intersection with $A_{Q,X}$.
\begin{figure}
\begin{center}\includegraphics[width=.8\textwidth]{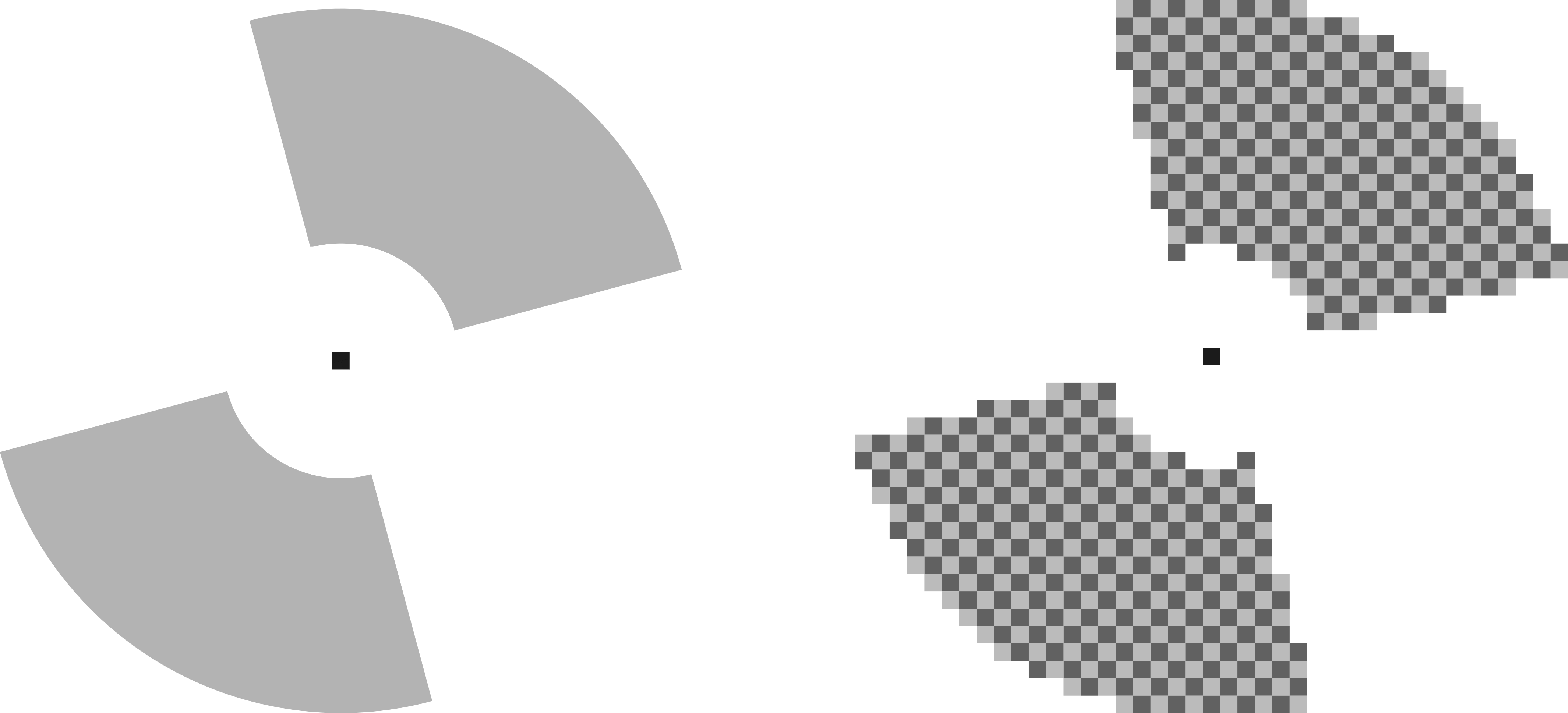}\end{center}
\caption{A conical annulus $A_{Q,X(V,\alpha)}$ (on the left) and its discretization $\Delta^*_{Q,X(V,\alpha)}$ (on the right), where $V=\Span \vec v$, $\vec v=(\cos(150^\circ),\sin(150^\circ))$, and $\alpha=1$.}
\end{figure} Of course, the discretized conical annulus covers the conical annulus, i.e.~$A_{Q,X}\subseteq\bigcup\Delta^*_{Q,X}$. In addition, define the \emph{dual discretized conical annulus} $\nabla^*_{R,X}$ by setting $$Q\in\nabla^*_{R,X}\quad\text{if and only if}\quad R\in \Delta^*_{Q,X}.$$ (In fact, it can be shown that $\nabla^*_{R,X}$ and $\Delta^*_{R,X}$ are the same set of dyadic cubes\footnote{We thank an anonymous referee for this observation.}, but $\nabla^*_{R,X}$ and $\Delta^*_{R,X}$ play different logical roles in the proofs below, so we use separate notation.) Note that each discretized region $\Delta^*_{Q,X}$ and $\nabla^*_{R,X}$ is a finite family of cubes with cardinality controlled by $n$ and $\alpha$. For every Radon measure $\mu$ on $\RR^n$, dyadic cube $Q$ with $\mu(Q)>0$, and bad cone $X=X(V,\alpha)$ with $\alpha\in(0,\infty)$, we define the \emph{conical defect} to be the quantity $$\Defect(\mu,Q,X)=\sum_{R\in\Delta^*_{Q,X}}\frac{\mu(R)}{\mu(Q)}\in[0,\infty).$$ We also set $\Defect(\mu,Q,X)=0$ if $\mu(Q)=0$. The conical defect is a weighted measurement of the mass of $\mu$ in the annular region $\bigcup \Delta^*_{Q,X}\supseteq A_{Q,X}$. It is anisotropic insofar as the normalization of each term $\mu(R)$ that appears in the defect depends on the value of the measure in cubes $Q\in\nabla^*_{R,X}$ emanating in different directions from the cube $R$.

\subsection*{Conical Dini Functions} For every Radon measure $\mu$ on $\RR^n$ and bad cone $X=X(V,\alpha)$ with $\alpha\in(0,\infty)$, we define the \emph{conical Dini function} $G_{\mu,X}:\RR^n\rightarrow[0,\infty]$, $$G_{\mu,X}(x)=\sum_{\side Q\leq 1} \Defect(\mu,Q,X)\,  \chi_Q(x)\quad\text{for all }x\in\RR^n,$$ where the sum ranges over all dyadic cubes $Q$ of side length at most 1 that contain $x$. The definition of $G_{\mu,X}(x)$ is similar in spirit to the definition of the density-normalized square functions in \cite{BS1,BS2,BS3}. The magnitude of the conical Dini function determines the interaction of $\mu$ with $m$-dimensional Lipschitz graphs. There are several possible ways to formulate this. Perhaps the most important is the following.

\begin{theorem}[Main Theorem] \label{t:main} Let $1\leq m\leq n-1$ be integers. Every Radon measure $\mu$ on $\RR^n$ decomposes uniquely as $\mu=\mu_G+\mu_G^\perp$, where $\mu_G$ is a Radon measure carried by $m$-dimensional Lipschitz graphs and $\mu_G^\perp$ is a Radon measure singular to $m$-dimensional Lipschitz graphs. The component measures are identified by \begin{align*}\mu_G&=\mu\res\{x\in\RR^n:G_{\mu,X}(x)<\infty \text{ for some bad cone }X\},\\
\mu_G^\perp &=\mu\res\{x\in\RR^n:G_{\mu,X}(x)=\infty \text{ for every bad cone }X\}.\end{align*} That is, there exists a sequence $\Gamma_1,\Gamma_2,\dots$ of $m$-dimensional Lipschitz graphs such that $\mu_G(\RR^n\setminus\bigcup_1^\infty\Gamma_i)=0$ and $\mu_G^\perp(\Gamma)=0$ for every $m$-dimensional Lipschitz graph $\Gamma$.
\end{theorem}

The main theorem implies that to determine whether or not a measure charges some Lipschitz graph or is carried by Lipschitz graphs it is enough to \emph{evaluate the measure on only countably many sets}, e.g.~on dyadic cubes of side length at most 1.

\subsection*{Consequences} The first two corollaries are immediate applications of the main theorem. For variations on Corollary \ref{carried-by-graphs} and \ref{charge-a-graph}, which account for the direction $V$ and Lipschitz constant $\alpha$ of the underlying Lipschitz graphs, see \S\S\,2 and 3.

\begin{corollary}\label{carried-by-graphs} A Radon measure $\mu$ on $\RR^n$ is carried by $m$-dimensional Lipschitz graphs if and only if $G_{\mu,X}(x)<\infty$ for some bad cone $X=X(V_x,\alpha_x)$ for $\mu$-a.e.~$x\in\RR^n$.\end{corollary}

\begin{corollary}\label{charge-a-graph}A Radon measure $\mu$ on $\RR^n$ charges some $m$-dimensional Lipschitz graph, i.e.~$\mu(\Gamma)>0$ for some Lipschitz graph $\Gamma$, if and only if there exists $E\subseteq\RR^n$ with $\mu(E)>0$ such that $G_{\mu,X}(x)<\infty$ for some bad cone $X=X(V_x,\alpha_x)$ for each $x\in E$.\end{corollary}

A basic geometric measure-theoretic fact is that every $m$-dimensional Lipschitz graph in $\RR^n$ has locally finite $m$-dimensional packing measure. Further, typical points of sets of finite $s$-dimensional packing measure have positive lower $s$-dimensional density. Hence we discover the following relationship between the conical Dini functions for $\mu$ and the lower $m$-dimensional density for $\mu$. For details, see e.g.~\cite[Lemma 2.7, 2.8]{BS1}.

\begin{corollary}Let $\mu$ be a Radon measure on $\RR^n$. \begin{enumerate}
\item  At $\mu$-a.e.~$x\in\RR^n$ such that $G_{\mu,X}(x)<\infty$ for some bad cone $X$, we have $$\underline{D}^m(\mu,x)=\liminf_{r\downarrow 0}\frac{\mu(B(x,r))}{r^m}>0.$$
\item At $\mu$-a.e.~$x\in\RR^n$ such that $\underline{D}^m(\mu,x)=0$, we have $G_{\mu,X}(x)=\infty$ for every bad cone $X$.\end{enumerate}\end{corollary}

It is a difficult, open problem to obtain similar theorems for Radon measures and Lipschitz images of $\RR^m$ in $\RR^n$ when $2\leq m\leq n-1$. (However, the case when $\mu$ vanishes on sets of zero $m$-dimensional Hausdorff measure is well understood, see e.g.~\cite{Mattila}.) The obstacle is back at the beginning of the paper: there is no known substitute for Lemma \ref{geometric-lemma} that characterizes subsets of Lipschitz images. Even the case $m=2$, $n=3$ is wide open. For related work on subsets of alternative classes of higher-dimensional curves and surfaces, see \cite{AS-TST, BNV, ENV-Banach, Hyde-TST} and the references therein.

\subsection*{Organization} We establish sufficient conditions for Lipschitz graph rectifiability in \S\ref{sec:sufficient}, followed by necessary conditions in \S\ref{sec:necessary}. Using these results, we prove Theorem \ref{t:main} in \S \ref{sec:main}. Finally, we discuss variations on the main theorem in \S\ref{sec:variation}.

\section{Sufficient Conditions}\label{sec:sufficient} We adopt the following standard notation. We write $C=C(p,q,\dots)$ to denote that $0<C<\infty$ is a constant depending on at most the parameters $p,q,\dots$. The value of $C$ may change from line to line. The notation $a\lesssim_{p,q,...}\! b$ is short hand for $a\leq C(p,q,\dots)\,b$.

\begin{lemma}\label{bounded-geometry} Let $X=X(V,\alpha)$ be a bad cone over $V\in G(n,m)$ with opening $\alpha\in(0,\infty)$. In the definition of the conical defect, we may choose the radius $$r_{Q,X}=81\sqrt{n}\max(\alpha,1/\alpha)\side Q.$$ For every dyadic cube $Q$ and for every $x\in Q$,  $$X_x\cap B(x,s_{Q,X})\setminus U\left(x,\tfrac{1}{2}s_{Q,X}\right)\subseteq A_{Q,X},\quad \text{where }s_{Q,X}=r_{Q,X}-\sqrt{n}\side Q.$$ There exists a constant $C_1=C_1(n,\alpha)$ such that for every dyadic cube $Q$ and $R\in\Delta^*_{Q,X}$, the Hausdorff distance between $Q$ and $R$ is at most $C_1\side Q$. Moreover, there exists a constant $C_2=C_2(n,\alpha)<\infty$ such that $\Delta^*_{Q,X}$ and $\nabla^*_{R,X}$ have cardinality at most $C_2$ for every dyadic cube $Q$ and $R$ in $\RR^n$. \end{lemma}

\begin{proof} For ease of computation, pick coordinates on $\RR^n$ so that $Q=[-\frac{1}{2},\frac{1}{2})\times \cdots\times [-\frac{1}{2},\frac{1}{2})$ is a ``dyadic cube'' of side length 1 with center at the origin. We return to the conventional definition of dyadic cubes at the conclusion of the proof. Suppose that $R$ is another dyadic cube of side length 1 such that $R\cap X_Q\setminus U(0,r/3)\neq\emptyset$. We first want to determine how large $r$ must be to ensure that $R\cap U(0,r/4)=\emptyset$. Choose any point $z\in R\cap X_Q\setminus U(0,r/3)$. For any $x\in R$, $|x| \geq |z|-\diam R \geq r/3-\sqrt{n}.$ Hence $|x|>r/4$ if $r>12\sqrt{n}$. Impose this lower bound on $r$.

Next we find how large $r$ must be to guarantee that $\gap(R, X_b(V,\alpha/2)^c)\geq \diam Q$ for every $b\in Q$. Continuing to work with $z$, pick $a\in Q$ such that $z\in X_a$, which exists because $z\in X_Q$. Fix $b\in Q$. To show that $\gap(R,X_b(V,\alpha/2)^c)\geq \diam Q$, it suffices to prove $B(z,2\diam R)\subseteq X_b(V,\alpha/2)$. Fix $w\in B(z,2\diam R)$ and recall that $\diam Q=\diam R=\sqrt{n}$. By repeated use of the triangle inequality and definition of $X_a$, \begin{equation*}\begin{split}
\dist(w,V_b)\geq \dist(z,V_a)-3\sqrt{n}> \alpha&\dist(z,V_a^\perp)-3\sqrt{n}\geq \alpha\dist(w,V_b^\perp) - 3(1+\alpha)\sqrt{n} \\ &\ \ =\frac{\alpha}{2}\dist(w,V_b^\perp)+\frac{\alpha}{2}\dist(w,V_b^\perp)-3(1+\alpha)\sqrt{n}.\end{split}\end{equation*} We now split into cases. From the displayed inequality, we see that $w\in X_b(V,\alpha/2)$ if $(\alpha/2)\dist(w,V_b^\perp)\geq 3(1+\alpha)\sqrt{n}$. Suppose otherwise that $\dist(w,V_b^\perp)< 6\sqrt{n}(1+\alpha)/\alpha$. By the Pythagorean theorem, we have \begin{align*}\dist(w,V_b)^2=|w-b|^2-\dist(w,V_b^\perp)^2 &\geq (|z|-|w-z|-|b|)^2-\dist(w,V_b^\perp)^2\\
&\geq (r/3-3\sqrt{n})^2-\dist(w,V_b^\perp)^2.\end{align*} We want $\dist(w,V_b)>(\alpha/2)\dist(w,V_b^\perp)$. Thus, we want $r$ to be large enough so that $$\left(\frac{r}{3}-3\sqrt{n}\right)^2\geq \left[\left(\frac{\alpha}{2}\right)^2+1\right]\dist(w,V_b^\perp)^2.$$ Set $\rho=r/3-3\sqrt{n}$. Because $\dist(w,V_b^\perp)< 6\sqrt{n}(1+\alpha)/\alpha$ and $(\frac{1}{4}\alpha^2+1)<(1+\alpha)^2$, it suffices to choose $\rho$ large enough so that $\rho^2\geq 36n(1+\alpha)^4/\alpha^2.$ Taking square roots and using the crude estimate $(1+\alpha)^2/\alpha \le 4\max(\alpha,1/\alpha)$, we see that it suffices to assume that $r\geq 81\sqrt{n}\max(\alpha,1/\alpha).$

The computations above were for a dyadic cube $Q$ of side length 1 centered at the origin. By scale and translation invariance, we conclude that if we define $$r_{Q,X}=81\sqrt{n}\max(\alpha,1/\alpha)\side Q\quad\text{for all $Q$},$$ then $R\cap B(x_Q,r_{Q,X}/4)=\emptyset$ and $\gap(R,X_{x}(V,\alpha/2)^c)\geq \diam Q$ for all $R\in\Delta^*_Q$ and $x\in Q$. Now, put $s_{Q,X}=r_{Q,X}-\sqrt{n}\side Q$ for each dyadic cube $Q$. Fix a dyadic cube $Q$ and $a\in Q$. Suppose that $y\in X_a\cap B(a,s_{Q,X})\setminus U(a,\frac12 s_{Q,X})$. On the one hand, $$|y-x_Q| \leq |y-a|+|a-x_Q| \leq s_{Q,X}+\sqrt{n}\side Q= r_{Q,X}.$$ On the other hand, \begin{align*}|y-x_Q| \geq |y-a|-|a-x_Q| &\geq \frac{s_{Q,X}}{2}-\sqrt{n}\side Q \\ &\geq \frac{r_{Q,X}}{2} - \frac{3}{2}\sqrt{n}\side Q \geq  39\sqrt{n}\max(\alpha,1/\alpha)\side Q>\frac{r_{Q,X}}{3}.\end{align*} Thus, $y\in X_a\cap B(x_Q,r_{Q,X})\setminus U(x_Q,r_{Q,X}/3)\subseteq A_{Q,X}$.

For every dyadic cube $Q$ and $R\in\Delta^*_{Q,X}$, we have $Q\subseteq B(x_Q,\sqrt{n} \side Q)$ and $R\subseteq B(x_Q,r_{Q,X}+\sqrt{n}\side Q)$, because $R$ intersects $A_{Q,X}$ and $\diam Q=\sqrt{n}\side Q$. Therefore, the Hausdorff distance between $Q$ and $R$ is at most $C_1(n,\alpha)\side Q$. By volume doubling, it follows that $\Delta^*_{Q,X}$ and $\nabla^*_{R,X}$ have cardinality at most $C_2(n,\alpha)$.
\end{proof}

We say that $\mathcal{T}$ is a \emph{tree of dyadic cubes} if $\mathcal{T}$ is a set of dyadic cubes ordered by inclusion such that $\mathcal{T}$ has a unique maximal element, denoted by $\Top(\mathcal{T})$, and if $Q\in\mathcal{T}$, then $P\in\mathcal{T}$ for all dyadic cubes $Q\subseteq P\subseteq\Top(\mathcal{T})$. We may partition $\mathcal{T}=\bigcup_0^\infty\mathcal{T}_i$, where $$\side Q=2^{-i}\side\Top(\mathcal{T})\quad\text{for all }Q\in\mathcal{T}_i.$$ An \emph{infinite branch} is a decreasing sequence $Q_0\supseteq Q_1\supseteq Q_2\supseteq\cdots$ of cubes in $\mathcal{T}$ with each $Q_i\in\mathcal{T}_i$. We define the \emph{set of leaves} of $\mathcal{T}$, denoted by $\Leaves(\mathcal{T})$, to be $$\Leaves(\mathcal{T})=\bigcup\left\{\bigcap_{i=0}^\infty Q_i: Q_0\supseteq Q_1\supseteq Q_2\supseteq\cdots\text{ is an infinite branch of $\mathcal{T}$}\right\}.$$ Although the set of leaves is facially a union over uncountably many infinite branches, because $\#\mathcal{T}_i<\infty$ for all $i\geq 0$, one may prove that $\Leaves(\mathcal{T})=\bigcap_{i=0}^\infty \bigcup\mathcal{T}_i$; e.g., see the argument at the top of \cite[p.~48]{Rogers}. Hence $\Leaves(\mathcal{T})$ is an $F_{\sigma\delta}$ Borel set.

\begin{lemma}\label{basic-lemma} Let $\mathcal{T}$ be a tree of dyadic cubes in $\RR^n$, let $\mu$ be a Radon measure on $\RR^n$, and let $X=X(V,\alpha)$ be a bad cone. If $\mu(R)=0$ for every $Q\in\mathcal{T}$ and $R\in \mathcal{T}\cap \Delta^*_{Q,X}$, then there is a Lipschitz function $f:V\rightarrow V^\perp$ with Lipschitz constant at most $\alpha$ such that $\mu(\Leaves(\mathcal{T})\setminus\Graph(f))=0$.\end{lemma}

\begin{proof}We may assume that $\Leaves(\mathcal{T})\neq\emptyset$, since otherwise the conclusion is trivial. Let $A=\Leaves(\mathcal{T})\setminus \bigcup_{Q\in\mathcal{T}}\bigcup_{R\in\mathcal{T}\cap \Delta^*_{Q,X}} R$.  Then $$\mu(\Leaves(\mathcal{T})\setminus A) \leq \sum_{Q\in\mathcal{T}}\sum_{R\in\mathcal{T}\cap\Delta^*_{Q,X}}\mu(R)=0.$$ We will use Lemma \ref{geometric-lemma} to show that $A$ is contained in the graph of a Lipschitz function over $V$ with Lipschitz constant at most $\alpha$. Let $x\in A$ and pick an infinite branch $Q_0\supseteq Q_1\supseteq Q_2\supseteq\cdots$ such that $\{x\}=\bigcap_0^\infty Q_i$. We must show that $A\cap X_x=\emptyset$. For each dyadic cube $Q$, write $s_{Q,X}=r_{Q,X}-\sqrt{n}\side Q$ and note that $s_{Q,X}=C(n,\alpha)\side Q\gg \diam Q$. By Lemma \ref{bounded-geometry}, for each cube $Q_i$ in the infinite branch containing $x$, $$A\cap X_x\cap B(x,s_{Q_i,X})\setminus U\left(x,\tfrac{1}{2}s_{Q_i,X}\right)\subseteq A\cap A_{Q_i,X}\subseteq \bigcup \mathcal{T}\cap\Delta^*_{Q_i,X}\subseteq \RR^n\setminus A.$$ Because $s_{Q_{i+1},X}=\frac12 s_{Q_i,X}$ for each $i\geq 0$, it follows that $A\cap X_x\cap B(x,s_{Q_0,X})=\emptyset$. Also, since $A\subseteq Q_0$ and $s_{Q_0,X}\gg \diam Q_0$, we have $A\cap X_x\setminus B(x,s_{Q_0,X})=\emptyset$, as well. Thus, $A\cap X_x=\emptyset$ for all $x\in A$. By Lemma \ref{geometric-lemma}, there exists $f:V\rightarrow V^\perp$ with Lipschitz constant at most $\alpha$ such that $A\subseteq\Graph(f)$. Finally, \begin{equation*}\mu(\Leaves(\mathcal{T})\setminus\Graph(f))\leq \mu(\Leaves(\mathcal{T})\setminus A)=0.\qedhere\end{equation*}
\end{proof}

We have reached the main technical result of the paper. The proof of Proposition \ref{build-graphs} dictates the definition of the conical defect; see especially the computation in \eqref{e:main-computation}.

\begin{proposition}[Drawing Lipschitz Graphs through Leaves of a Tree] \label{build-graphs} Let $\mathcal{T}$ be a tree of dyadic cubes in $\RR^n$ and let $\mu$ be a Radon measure on $\RR^n$. If there exists a bad cone $X=X(V,\alpha)$ such that $\sum_{Q\in\mathcal{T}} \Defect(\mu,Q,X)\,\mu(Q)<\infty$, then $\mu\res\Leaves(\mathcal{T})$ is carried by graphs of Lipschitz functions $f:V\rightarrow V^\perp$ with Lipschitz constant at most $\alpha$.\end{proposition}

\begin{proof} Suppose that $\mathcal{T}=\bigcup_0^\infty \mathcal{T}_i$, where $\mathcal{T}_i$ denotes the cubes of side length $2^{-i}\side\Top(\mathcal{T})$. Without loss of generality, we may assume that $\side\Top(\mathcal{T})=1$ and $\mu(Q)>0$ for every cube $Q\in \mathcal{T}$, because deleting cubes in the tree with $\mu$ measure zero has no effect on the graph rectifiability of $\mu\res\Leaves(\mathcal{T})$. As every $\mu$ null set is trivially graph rectifiable, we may further assume that $\mu(\Leaves(\mathcal{T}))>0$.  The general scheme of the proof is to identify countably many subtrees of $\mathcal{T}$ whose sets of leaves are each contained in a Lipschitz graph and collectively cover $\mu$ almost all of the set of leaves of $\mathcal{T}$.

We will say that a dyadic cube $R\in\mathcal{T}$ is \emph{bad} if there exists $Q\in\mathcal{T}$ such that $R\in\Delta^*_{Q,X}$. Bad cubes are the obstacles to invoking Lemma \ref{basic-lemma}. Let us compute the total measure of bad cubes in level $i$ of the tree. There are no bad cubes in $\mathcal{T}_0$, because the first level only contains one cube. Fix $i\geq 1$ and let $\mathcal{B}_i$ denote the set of $Q\in\mathcal{T}_i$ such that there exists $R\in\Delta^*_{Q,X}\cap\mathcal{T}$, i.e.~the discretized conical annulus for $Q$ contains a bad cube. Then \begin{equation}\begin{split}\label{e:main-computation}\sum_{\text{bad }R\in\mathcal{T}_i}\mu(R) &=\sum_{\text{bad }R\in\mathcal{T}_i}\sum_{Q\in\nabla^*_{R,X}\cap\mathcal{T}}\frac{\mu(R)\mu(Q)}{\mu(\bigcup \nabla^*_{R,X}\cap\mathcal{T})}  = \sum_{Q\in\mathcal{B}_i} \sum_{R\in\Delta^*_{Q,X}\cap\mathcal{T}}\frac{\mu(R)\mu(Q)}{\mu(\bigcup \nabla^*_{R,X}\cap\mathcal{T})} \\ &\leq \sum_{Q\in\mathcal{B}_i} \sum_{R\in\Delta^*_{Q,X}}\frac{\mu(R)\mu(Q)}{\mu(Q)}\leq \sum_{Q\in\mathcal{T}_i}\Defect(\mu,Q,X)\,\mu(Q),\end{split}\end{equation} where in the penultimate inequality $0<\mu(Q)\leq \mu(\bigcup\nabla^*_{R,X}\cap \mathcal{T})$ because $Q\in\mathcal{B}_i$ and $Q\in\nabla^*_{R,X}$ for all $R\in\Delta^*_{Q,X}$. The first equality in \eqref{e:main-computation} may be interpreted as equitably distributing the mass of a bad cube $R$ to the cubes $Q\in\nabla^*_{R,X}\cap\mathcal{T}$ which ``see'' $R$.

Let $0<\delta<1$ be given. Because the weighted sum of the conical defect over $\mathcal{T}$ converges, there exists $i_0=i_0(\delta)\geq 1$ sufficiently large such that the tail \begin{equation}\label{bad-cube-sum} \sum_{i=i_0}^\infty \sum_{\text{bad }R\in\mathcal{T}_i} \mu(R) \leq \sum_{i=i_0}^\infty \sum_{Q\in\mathcal{T}_i}\Defect(\mu,Q,X)\,\mu(Q)< \delta\,\mu(\Leaves(\mathcal{T})).\end{equation} Let $Q^\delta_1,\dots,Q^\delta_k$ be an enumeration of the cubes in $\mathcal{T}_{i_0}$ such that each cube $Q^\delta_j$ is not bad. Then let $\mathcal{U}^\delta_1,\dots,\mathcal{U}^\delta_k$ denote the maximal subtrees of $\mathcal{T}$ with $\Top(\mathcal{U}^\delta_j)=Q^\delta_j$ that contain no bad cubes. By \eqref{bad-cube-sum}, the trees exist, and the set $A^{\delta}=\bigcup_{j=1}^k \Leaves(\mathcal{U}^\delta_j)$ satisfies $$\mu(A^{\delta})\geq (1-\delta) \mu(\Leaves(\mathcal{T})).$$ Moreover, by $k$ applications of Lemma \ref{basic-lemma} (each $\mathcal{U}^{\delta}_j$ contains no bad cubes), $A^{\delta}$ is contained in the union of $k=k(\delta)$ Lipschitz graphs over $V$ of Lipschitz constant at most $\alpha$.

To complete the proof, repeat the construction in the previous paragraph over any countable choice of parameters $\delta=\delta_j$ with $\lim_{j\rightarrow\infty}\delta_j=0$.
\end{proof}

Let $\mathcal{T}$ be a tree of dyadic cubes, let $b:\mathcal{T}\rightarrow[0,\infty)$ be any function, and let $\mu$ be a Radon measure on $\RR^n$. Following \cite[\S5]{BS3}, we define the \emph{$\mu$-normalized sum function} $$S_{\mathcal{T},b}(\mu,x)=\sum_{Q\in\mathcal{T}} b(Q)\frac{\chi_Q(x)}{\mu(Q)}\quad\text{for all }x\in\RR^n,$$ with the convention that $0/0=0$ and $1/0=\infty$. For example, for every dyadic cube $Q_0$ of side length 1, the conical Dini function $G_{\mu,X}(x)=S_{\mathcal{T},b}(\mu,x)$ for all $x\in Q_0$, where $\mathcal{T}$ is the tree of dyadic cubes contained in $Q_0$ and $b(Q)=\Defect(\mu,Q,X)\,\mu(Q)$.

\begin{lemma}[Localization Lemma, {\cite[Lemma 5.6]{BS3}}] \label{localization} Let $\mathcal{T}$ be a tree of dyadic cubes, let $b:\mathcal{T}\rightarrow[0,\infty)$, and let $\mu$ be a Radon measure on $\RR^n$. For all $N<\infty$ and $\varepsilon>0$, there exists a partition of $\mathcal{T}$ into a set $\mathcal{G}$ of \emph{good cubes} and a set $\mathcal{B}$ of \emph{bad cubes} with the following properties. \begin{enumerate}
\item Either $\mathcal{G}=\emptyset$ or $\mathcal{G}$ is a tree of dyadic cubes with $\Top(\mathcal{G})=\Top(\mathcal{T})$.
\item Every child of a bad cube is a bad cube: if $P,Q\in\mathcal{T}$, $P\in\mathcal{B}$, and $Q\subseteq P$, then $Q\in\mathcal{B}$.
\item The set $A=\{x\in\Top(\mathcal{T}):S_{\mathcal{T},b}(x)\leq N\}$ is Borel and $$\mu(A\cap\Leaves(\mathcal{G}))\geq (1-\varepsilon\mu(\Top(\mathcal{T})))\,\mu(A).$$
\item The sum of $b$ over $\mathcal{G}$ is finite: $\sum_{Q\in\mathcal{G}} b(Q)<N/\varepsilon$.
\end{enumerate}
\end{lemma}

Countably many applications of Proposition \ref{build-graphs} and Lemma \ref{localization} yield the following sufficient condition  in terms of the conical Dini function for Lipschitz graph rectifiability of a measure with prescribed direction $V$ and Lipschitz constant $\alpha$. Being similar to the proof of \cite[Theorem 5.1]{BS3}, we omit the details.

\begin{theorem}\label{t:suff} Let $\mu$ be a Radon measure on $\RR^n$ and let $X=X(V,\alpha)$ be a bad cone for some $V\in G(n,m)$ and $\alpha\in(0,\infty)$. Then $\mu\res\{x\in\RR^n:G_{\mu,X}(x)<\infty\}$ is carried by graphs of Lipschitz functions $f:V\rightarrow V^\perp$ of Lipschitz constant at most $\alpha$.
\end{theorem}

\section{Necessary Conditions}\label{sec:necessary}

Recall that $\gap(S,T)=\inf\{|s-t|:s\in S, t\in T\}$ for all nonempty sets $S,T\subseteq\RR^n$. We define the quantity $\excess(S,T)=\sup_{s\in S}\inf_{t\in T}|s-t|\in[0,\infty]$ for all nonempty sets $S,T\subseteq\RR^n$. By convention, we also set $\excess(\emptyset,S)=0$, but leave $\excess(S,\emptyset)$ undefined. The Hausdorff distance between nonempty sets $S$ and $T$ is defined to be the maximum of $\excess(S,T)$ and $\excess(T,S)$.

To establish necessary conditions for Lipschitz graph rectifiability in terms of the conical Dini functions, we follow the strategy used in \cite[\S3]{BS1} and \cite[\S4]{BS3} to prove necessary conditions for a Radon measure to be carried by rectifiable curves. The argument must be modified to incorporate the geometry of Lipschitz graphs.

\begin{proposition}\label{integral-bound} Let $\mu$ be a Radon measure on $\RR^n$ and let $X=X(V,\alpha)$ be a bad cone for some $V\in G(n,m)$ and $\alpha\in(0,\infty)$. Suppose that $\Gamma=\Graph(f)$ for some Lipschitz function $f:V\rightarrow V^\perp$ with Lipschitz constant at most $\alpha/2$. There exists a constant $C=C(n,\alpha)>1$ such that for every $x_0\in\Gamma$ and $r_0>0$, \begin{equation}\label{e:integral} \int_{\Gamma\cap B(x_0,r_0)} G_{\mu,X}(x)\,d\mu(x) \lesssim_{n,\alpha} \mu(B(x_0,r_0+C)\setminus \Gamma)<\infty.\end{equation} In particular, $G_{\mu,X}(x)<\infty$ at $\mu$-a.e.~$x\in\Gamma$.\end{proposition}

\begin{proof} Let $\mu$, $V$, $\alpha$, $f$, $\Gamma$, $x_0$, and $r_0$ be fixed as in the statement of the lemma. Abbreviate $\Gamma_0=\Gamma\cap B(x_0,r_0)$. By Tonelli's theorem, \begin{align*}\int_{\Gamma_0} G_{\mu,X}(x)\,d\mu(x) &= \sum_{\side Q\leq 1} \Defect(\mu,Q,X)\int_{\Gamma_0}\chi_Q(x)\,d\mu(x) \\ &=\sum_{\side Q\leq 1} \Defect(\mu,Q,X)\,\mu(\Gamma_0\cap Q)\leq \sum_{\stackrel{\side Q\leq 1}{\mu(\Gamma_0\cap Q)>0}} \Defect(\mu,Q,X)\,\mu(Q).\end{align*} For any dyadic cube $Q$ such that $\side Q\leq 1$ and $\mu(\Gamma_0\cap Q)>0$, $$\Defect(\mu,Q,X)\,\mu(Q)=\sum_{R\in\Delta^*_{Q,X}}\frac{\mu(R)}{\mu(Q)}\mu(Q)=\mu\left(\bigcup \Delta^*_{Q,X}\right),$$ where $\bigcup\Delta^*_{Q,X}$ denotes the union of all cubes in $\Delta^*_{Q,X}$ Thus, $$\int_{\Gamma_0} G_{\mu,X}(x)\,d\mu(x) \leq \sum_{\stackrel{\side Q\leq 1}{\mu(\Gamma_0\cap Q)>0}} \mu\left(\bigcup \Delta^*_{Q,X}\right).$$ We now aim to prove that the \emph{non-tangential regions} $T_{Q,X}=\bigcup \Delta^*_{Q,X}$ associated to dyadic cubes with $\side Q\leq 1$ and $\mu(\Gamma_0\cap Q)>0$ are contained in $\RR^n\setminus\Gamma$ and have bounded overlap. This requires that we use the geometry of the Lipschitz graph $\Gamma$.

Let $Q$ be a dyadic cube of side length at most 1 such that $\mu(\Gamma_0\cap Q)>0$. Pick $a\in \Gamma_0\cap Q$. Because $\Gamma$ is the graph of a Lipschitz function over $V$ with Lipschitz constant at most $\alpha/2$, Lemma \ref{geometric-lemma} tells us that the graph $\Gamma$ is contained in $X_a(V,\alpha/2)^c$. By definition of $r_{Q,X}$ or proof of Lemma \ref{bounded-geometry}, $\gap(R,X_a(V,\alpha/2)^c)\geq \diam Q$ for all $R\in\Delta^*_{Q,X}$. Hence \begin{equation}\label{e:gap-below} \gap(T_{Q,X},\Gamma)\geq \gap(T_{Q,X},X_a(V,\alpha/2)^c)\geq \diam Q.\end{equation} If $R\in\Delta^*_{Q,X}$, then there exists $z\in R$ such that $|z-x_Q|\leq r_{Q,X}=81\max(\alpha,1/\alpha)\diam Q$. For an arbitrary point $y\in R$, $|y-a|\leq |y-z|+|z-x_Q|+|x_Q-a|$. Thus, \begin{equation}\label{e:excess-above} \excess(T_{Q,X},\Gamma) \leq \excess(T_{Q,X},\{a\}) \leq 83\max(\alpha,1/\alpha)\diam Q.\end{equation} It follows that $T_{Q,X}\subseteq B(x_0,r_0+83\sqrt{n}\max(\alpha,1/\alpha))\setminus \Gamma$. Furthermore, suppose that $Q'$ is a dyadic cube of side length $2^{-N}\side Q$ such that $\mu(\Gamma_0\cap Q')>0$. By \eqref{e:gap-below} and \eqref{e:excess-above}, $T_{Q,X}$ and $T_{Q',X}$ are disjoint if $2^{-N}83\max(\alpha,1/\alpha)<1$. Thus, if $T_{Q,X}\cap T_{Q',X}\neq\emptyset$, where $T_{Q,X}$ and $T_{Q',X}$ are non-tangential regions associated to dyadic cubes $Q$ and $Q'$ intersecting $\Gamma_0$ of side lengths $2^{-\lambda}$ and $2^{-\lambda'}$ at most 1, then $|\lambda-\lambda'|\leq C(n,\alpha)$. Another consequence of \eqref{e:excess-above} is that $\diam T_{Q,X} \leq  166\max(\alpha,1/\alpha)\diam Q$. It follows that we have bounded overlap of the non-tangential regions: $T_{Q,X}$ intersects $T_{Q',X}$ for at most $C(n,\alpha)$ other cubes $Q'$ with $\side Q'\leq 1$ and $\mu(\Gamma_0\cap Q')>0$. Therefore, $$\int_{\Gamma_0} G_{\mu,X}(x)\,d\mu(x) \leq \sum_{\stackrel{\side Q\leq 1}{\mu(\Gamma_0\cap Q)>0}} \mu\left(T_{Q,X}\right)\lesssim_{n,\alpha} \mu(B(x_0,r_0+83\sqrt{n}\max(\alpha,1/\alpha))\setminus \Gamma).$$ The last displayed quantity is finite, because $\mu$ is a Radon measure. This verifies that \eqref{e:integral} holds for all $x_0\in \Gamma$ and $r_0>0$. Hence the conical Dini function $G_{\mu,X}(x)<\infty$ at $\mu$-a.e.~$x\in \Gamma_0=\Gamma\cap B(x_0,r_0)$. Because $r_0>0$ was arbitrary and $\Gamma \subseteq \bigcup_{k=1}^\infty B(x_0,k)$,  we conclude that $G_{\mu,X}(x)<\infty$ at $\mu$-a.e.~$x\in\Gamma$.
\end{proof}

\begin{theorem}Let $\mu$ be a Radon measure on $\RR^n$. Suppose $V\in G(n,m)$ and $\alpha\in(0,\infty)$. If $\mu$ is carried by graphs of Lipschitz functions $f:V\rightarrow V^\perp$ of Lipschitz constant at most $\alpha$ and $\beta\geq 2\alpha$, then $G_{\mu,X(V,\beta)}(x)<\infty$ at $\mu$-a.e.~$x\in\RR^n$.\end{theorem}

\begin{proof} The hypothesis asserts that there exist Lipschitz functions $f_1,f_2,\dots:V\rightarrow V^\perp$ with Lipschitz constant at most $\alpha$ such that $\mu(\RR^n\setminus \bigcup_1^\infty \Gamma_i)=0$, where each $\Gamma_i=\Graph(f_i)$. Let $\beta\geq 2\alpha$, so that each function $f_i$ has Lipschitz constant at most $\beta/2$. By Proposition \ref{integral-bound}, for each $i\geq 1$, there exists a set $N_i\subseteq\Gamma_i$ such that $\mu(N_i)=0$ and $G_{\mu,X(V,\beta)}(x)<\infty$ at every $x\in\Gamma_i\setminus N_i$. Set $E=\bigcup_1^\infty \Gamma_i\setminus N_i$. Then $G_{\mu,X(V,\beta)}(x)<\infty$ for every $x\in E$ and \begin{equation*}\mu(\RR^n\setminus E) \leq \mu\left(\RR^n\setminus \bigcup_1^\infty\Gamma_i\right)+\sum_1^\infty \mu(N_i)=0. \qedhere\end{equation*}
\end{proof}

\section{Proof of the Main Theorem}\label{sec:main}

Let $1\leq m\leq n-1$ and let $\mu$ be a Radon measure on $\RR^n$. Existence and uniqueness of the decomposition is standard.

\begin{lemma}\label{decomposition-lemma} There exists a unique decomposition $\mu=\mu_G+\mu_G^\perp$, where $\mu_G$ is a Radon measure that is carried by $m$-dimensional Lipschitz graphs and $\mu_G^\perp$ is a Radon measure that is singular to $m$-dimensional Lipschitz graphs.\end{lemma}

\begin{proof} The decomposition follows from a simple modification of the usual proof of the Lebesgue decomposition theorem. For each integer $r\geq 2$, use the approximation property of the supremum to choose a set $\Gamma_r$, which is a finite union of $m$-dimensional Lipschitz graphs, such that $$\mu(\Gamma_r)\geq (1-1/r)\sup_{\Gamma}\mu(\Gamma\cap B(0,r))<\infty,$$ where the supremum ranges over all sets $\Gamma$, which are finite unions of $m$-dimensional Lipschitz graphs in $\RR^n$. Then $\mu_G=\mu\res\bigcup_{r=2}^\infty \Gamma_r$ and $\mu_G^\perp=\mu-\mu_G$. Uniqueness of the decomposition can be proved by contradiction. For full details, see e.g.~the appendix of \cite{BV}.\end{proof}

The content of Theorem \ref{t:main} over Lemma \ref{decomposition-lemma} and our task in the remainder of the proof is to identify the component measures $\mu_G$ and $\mu^\perp_G$ using the conical Dini functions.

\begin{lemma}\label{measurability} For every bad cone $X$, the conical Dini function $G_{\mu,X}$ is Borel measurable.\end{lemma}

\begin{proof} By definition, each conical Dini function $G_{\mu,X}$ is a countable linear combination of characteristic functions of Borel sets.\end{proof}

\begin{lemma}\label{countable-decomposition} There exists a countable family $\mathcal{X}(n,m)$ of bad cones (independent of $\mu$) such that $G_{\mu,X}(x)<\infty$ at some $x\in\RR^n$ for some bad cone $X$ if and only if $G_{\mu,X'}(x)<\infty$ for some $X'\in\mathcal{X}(n,m)$.\end{lemma}

\begin{proof} The value of a conical Dini function $G_{\mu,X}(x)$ is determined by the value of the measure on dyadic cubes $Q$ and $R\in\Delta^*_{Q,X}$, where $Q$ ranges over all dyadic cubse of side length at most 1 that contains $x$. The cubes belonging to $\Delta^*_{Q,X}$ for any particular $Q$ are completely determined by $x_Q$, $\side Q$, and the arrangement of cubes in $\Delta^*_{Q_0,X}$, where $Q_0=[0,1)\times\cdots\times[0,1).$ By Lemma \ref{bounded-geometry}, the cubes in $\Delta^*_{Q_0,X}$ are dyadic cubes of side length 1 contained in $B(x_{Q_0},C(n,\alpha))$, where for each integer $N\geq 1$, the constant $C(n,\alpha)$ is uniformly bounded for all $\alpha\in(1/N,N)$. Therefore, there are only countably many possible configurations of $\Delta^*_{Q,X}$. Define $\mathcal{X}(n,m)$ by including exactly one bad cone $X$ for each possible configuration of $\Delta^*_{Q_0,X}$.\end{proof}

We are ready to complete the proof of Theorem \ref{t:main}. Let $\mu_1$ and $\mu_2$ be the measures defined by \begin{align*}\mu_1&=\mu\res\{x\in\RR^n:G_{\mu,X}(x)<\infty \text{ for some bad cone }X\},\\
\mu_2 &=\mu\res\{x\in\RR^n:G_{\mu,X}(x)=\infty \text{ for every bad cone }X\}.\end{align*} By Lemma \ref{countable-decomposition}, we may alternatively express \begin{align*} \mu_1 &=\mu\res\{x\in\RR^n:G_{\mu,X}(x)<\infty\text{ for some }X\in\mathcal{X}(n,m)\},\\
\mu_2 &=\mu\res\{x\in\RR^n:G_{\mu,X}(x)=\infty\text{ for every }X\in\mathcal{X}(n,m)\}.\end{align*} In view of Lemma \ref{measurability}, we conclude that $\mu_1$ and $\mu_2$ are restrictions of a Radon measure to a Borel set. Hence $\mu_1$ and $\mu_2$ are Radon.

On the one hand, for each bad cone $X$, $\mu_{G,X}=\mu\res\{x\in\RR^n:G_{\mu,X}(x)<\infty\}$ is carried by $m$-dimensional Lipschitz graphs by Theorem \ref{t:suff}. Because $\mathcal{X}(n,m)$ is countable, $\mu_+=\sum_{X\in\mathcal{X}(n,m)}\mu_{G,X}$ is also carried by Lipschitz graphs.  By Lemma \ref{countable-decomposition}, $\mu_1\leq \mu_+$. Thus, $\mu_1$ is carried by $m$-dimensional Lipschitz graphs, because the dominant measure $\mu_+$ is carried by Lipschitz graphs.

On the other hand, let $\Gamma$ be an arbitrary $m$-dimensional Lipschitz graph, say that $\Gamma=\Graph(f)$ where $f$ is a Lipschitz function $f:V\rightarrow V^\perp$ with Lipschitz constant at most $\alpha$. By Proposition \ref{integral-bound}, $G_{\mu,X(V,2\alpha)}(x)<\infty$ at $\mu$-a.e.~$x\in\Gamma$. Because $\mu_2\leq \mu$, we have $G_{\mu,X(V,2\alpha)}(x)<\infty$ at $\mu_2$-a.e.~$x\in\Gamma$, as well. By definition, the measure $\mu_2$ vanishes on $\{x\in\RR^n:G_{\mu,X(V,2\alpha)}(x)<\infty\}$. Therefore, $\mu_2(\Gamma)=0$. Since $\Gamma$ was arbitrary, we conclude that $\mu_2$ is singular to $m$-dimensional Lipschitz graphs.

It is immediate from the definition of $\mu_1$ and $\mu_2$ that $\mu=\mu_1+\mu_2$. Since $\mu_1$ and $\mu_2$ are Radon measures, $\mu_1$ is carried by $m$-dimensional Lipschitz graphs, and $\mu_2$ is singular to $m$-dimensional Lipschitz graphs, we know that $\mu_1=\mu_G$ and $\mu_2=\mu_G^\perp$ by uniqueness of the decomposition in Lemma \ref{decomposition-lemma}. This completes the proof of Theorem \ref{t:main}.

\section{Variations}\label{sec:variation}

We conclude with some remarks on flexibility in the definition of the conical defect and variations on the main theorem. The restriction to half-open dyadic cubes is hidden in the proof of the localization lemma for the $\mu$-normalized sum function (see Lemma \ref{localization}). The lemma remains valid for any system of sets $\mathscr{A}$ with the property that if $\mathcal{T}$ is a tree of sets in $\mathscr{A}$ and $\mathcal{B}$ is a subset of $\mathcal{T}$ such that $A,B\in\mathcal{B}$ and $A\subseteq B$ implies $A=B$, then $\mathcal{B}$ has bounded overlap with constants independent of $\mathcal{B}$ and $\mathcal{T}$. Of course, half-open dyadic cubes, half-open triadic cubes, etc.~enjoy this property with bounded overlap 1. One could probably design a version of the conical defect and main theorem with Euclidean balls by using the Besicovitch covering theorem instead of the localization lemma.

The main theorem remains valid if one replaces the conical defect by the larger quantity $$\sum_{R\in \Delta^*_{Q,X}} \frac{\mu(S_R)}{\mu(Q)},$$ where $S_R$ is any Borel set containing $R$ with diameter at most $C\diam R$ for some constant $1\leq C<\infty$ independent of $R$. This change requires increasing the size of the radius $r_{Q,X}$ depending on the constant $C$ to ensure that $\gap(S_R,X(V,\gamma)^c)\gtrsim \diam Q$ for some $\gamma<\alpha$. Because the sets $S_R$ may overlap, this is a slight strengthening of the necessary condition for Lipschitz graph rectifiability.

If one prefers a construction where the radius $r_{Q,X}$ of the conical annulus is independent of the cone opening $\alpha$, this can be achieved at the cost of taking the cubes $R\in\Delta^*_{Q,X}$ to have side length smaller than $Q$. In particular, there exists $\tilde r_{Q,X}$ depending only on $n$ and a \emph{jump parameter} $J\in\mathbb{N}$ depending on $n$ and $\alpha$ with the following property. If $R$ is a dyadic cube of side length $2^{-J}\side Q$ that intersects $X_Q\cap B(x_Q,\tilde r_{Q,X})\setminus B(x_Q,\tilde r_{Q,X}/3)$, then $R\cap B(x_Q,\tilde r_{Q,X}/4)=\emptyset$ and $\gap(R,X(V,\alpha/2)^c)\geq \diam R$. The proofs of the sufficient and necessary conditions with this modification are essentially the same as above, although there is more bookkeeping involving $J$.

Every $m$-dimensional plane $x_0+V$ is a Lipschitz graph over $V$ with constant at most $\alpha$ for every $\alpha>0$. It follows that $\mu\res\{x\in \RR^n: \forall_{V\in G(n,m)}\exists_{\alpha>0}\, G_{\mu,X(V,\alpha)}=\infty\}$ is singular to affine $m$-dimensional planes. However, this cannot be directly used to characterize Radon measures that are carried by or singular to planes. We leave finding such a characterization as an open problem for future research.

\bibliography{radon-graphs}
\bibliographystyle{amsbeta}

\end{document}